\documentclass[12pt]{amsart}
\usepackage{amsmath,amssymb,amsfonts,amsthm,amsopn}
\usepackage{graphicx,tikz}

\newtheorem{tm}{Theorem}[section]
\newtheorem{lemma}[tm]{Lemma}
\newtheorem{prop}[tm]{Proposition}
\newtheorem{cor}[tm]{Corollary}

\newtheorem{theorem}{Theorem}[section]
\newtheorem{corollary}[theorem]{Corollary}

\newtheorem{proposition}[theorem]{Proposition}

\newcommand{\beqa}{\begin{eqnarray*}}
\newcommand{\eeqa}{\end{eqnarray*}}

\def\<{\left<}
\def\>{\right>}

\def\mv1{M_v^1}

\def\o{\omega}

\def\Ren{\mathbb{R}^d}

\def\Sn2{S_{2}(L^{2}(\Ren))}
\def\S1{S_{1}(L^{2}(\Ren))}
\def\sig00{\sigma_{0,0}}

\input{tcilatex}

\begin{document}
\title[Time-frequency Analysis of Born-Jordan Operators]{Time-frequency
Analysis of Born-Jordan Pseudodifferential Operators}
\author{Elena Cordero}
\address{Dipartimento di Matematica, Universit\`a di Torino, Dipartimento di
Matematica, via Carlo Alberto 10, 10123 Torino, Italy}
\email{elena.cordero@unito.it}
\thanks{}
\author{Maurice de Gosson}
\address{University of Vienna, Faculty of Mathematics,
Oskar-Morgenstern-Platz 1 A-1090 Wien, Austria}
\email{maurice.de.gosson@univie.ac.at}
\thanks{}
\author{Fabio Nicola}
\address{Dipartimento di Scienze Matematiche, Politecnico di Torino, corso
Duca degli Abruzzi 24, 10129 Torino, Italy}
\email{fabio.nicola@polito.it}
\thanks{}
\date{}

\begin{abstract}
Born-Jordan operators are a class of pseudodifferential operators arising as
a generalization of the quantization rule for polynomials on the phase space
introduced by Born and Jordan in 1925. The weak definition of such operators
involves the Born-Jordan distribution, first introduced by Cohen in 1966 as
a member of the Cohen class. We perform a time-frequency analysis of the
Cohen kernel of the Born-Jordan distribution, using modulation and Wiener
amalgam spaces. We then provide sufficient and necessary conditions for
Born-Jordan operators to be bounded on modulation spaces. We use modulation
spaces as appropriate symbols classes.
\end{abstract}

\maketitle




\section{Introduction}

In 1925 Born and Jordan \cite{bj} introduced for the first time a rigorous
mathematical explanation of the notion of ``quantization''. This rule was
initially restricted to polynomials as symbol classes but was later extended
to the class of tempered distribution $\mathcal{S}^{\prime }({\mathbb{R}^{2d}%
})$ \cite{bogetal,Cohen2}. Roughly speaking, a quantization is a rule which
assigns an operator to a function (called symbol) on the phase space ${%
\mathbb{R}^{2d}}$. The Born-Jordan quantization was soon superseded by the
most famous Weyl quantization rule proposed by Weyl in \cite{Weyl1927},
giving rise to the well-known Weyl operators (transforms) (see, e.g. \cite%
{Wongbook}).

Recently there has been a regain in interest in the Born-Jordan
quantization, both in Quantum Physics and Time-frequency Analysis \cite{deGossonLuef2015}. The
second of us has proved that it is the correct rule if one wants matrix and
wave mechanics to be equivalent quantum theories \cite{gofound}. Moreover,
as a time-frequency representation, the Born-Jordan distribution has been
proved to be better than the Wigner distribution since it damps very well
the unwanted ``ghost frequencies'', as shown in \cite{bogetal,turunen}. For
a throughout and rigorous mathematical explanation of these phenomena we
refer to \cite{cgn2} whereas \cite[Chapter 5]{auger} contains the relevant
engineering literature about the geometry of interferences and kernel design.

To be more specific, the (cross-)Wigner distribution of signals $f,g$ in the
Schwartz class $\mathcal{S}(\mathbb{R}^{d})$ is defined by 
\begin{equation}
W(f,g)(x,\omega )=\int_{\mathbb{R}^{d}}e^{-2\pi iy\omega }f(x+\frac{y}{2})\overline{g(x-\frac{y}{2})}\,dy.  \label{wigner}
\end{equation}
The Weyl operator $\limfunc{Op}\nolimits_{\mathrm{W}}(a)$ with symbol $a\in 
\mathcal{S}^{\prime }({\mathbb{R}^{2d}})$ can be defined in terms of the
Wigner distribution by the formula 
\begin{equation*}
\langle \limfunc{Op}\nolimits_{\mathrm{W}}(a)f,g\rangle =\langle
a,W(g,f)\rangle .
\end{equation*}
For $z=(x,\omega )$, consider the Cohen kernel 
\begin{equation}
\Theta (z):=\mathrm{sinc}(x\omega )=
\begin{cases}
\displaystyle\frac{\sin (\pi x\omega )}{\pi x\omega } & \mbox{for}\,  \omega x\neq 0 \\ 
1 & \mbox{for}\, \omega x=0.
\end{cases}
\label{sincxp}
\end{equation}
The (cross-)Born-Jordan distribution $Q(f,g)$ is then defined by 
\begin{equation}
Q(f,g)=W(f,g)\ast \Theta _{\sigma },\quad f,g\in \mathcal{S}(\mathbb{R}^{d}),
\label{BJ}
\end{equation}%
where $\Theta _{\sigma }$ is the symplectic Fourier transform of $\Theta$ recalled in \eqref{SFT} below. Likewise the Weyl operator, a Born-Jordan operator with
symbol $a\in \mathcal{S}^{\prime }({\mathbb{R}^{2d}})$ can be defined as 
\begin{equation}
\langle \limfunc{Op}\nolimits_{\mathrm{BJ}}(a)f,g\rangle =\langle
a,Q(g,f)\rangle \quad f,g\in \mathcal{S}(\mathbb{R}^{d}).  \label{BJOP}
\end{equation}%
Any pseudodifferential operator admits a representation in the Born-Jordan
form $\limfunc{Op}\nolimits_{\mathrm{BJ}}(a)$, as stated in \cite{cgn01}. \par
Now, a first relevant feature of this work is to have computed the Cohen kernel $\Theta
_{\sigma }$ explicitly (cf. the subsequent Proposition \ref{Thetasigma}). Namely \begin{equation*}
\Theta _{\sigma }(\zeta _{1},\zeta _{2})=
\begin{cases}
-2\,\mathrm{Ci}(4\pi |\zeta _{1}\zeta _{2}|),\quad (\zeta _{1},\zeta
_{2})\in \mathbb{R}^{2},\,d=1 \\ 
\mathcal{F}(\chi _{\{|s|\geq 2\}}|s|^{d-2})(\zeta _{1},\zeta _{2}),\quad
(\zeta _{1}\zeta _{2})\in {\mathbb{R}^{2d}},\,d\geq 2,%
\end{cases}
\end{equation*}
where $\chi _{\{|s|\geq 2\}}$ is the characteristic function of the set $\{s\in \mathbb{R}\,:\,|s|\geq 2\}$ and where
\begin{equation*}  \label{cosint}
\mathrm{Ci}(t)=-\int_t^{+\infty} \frac{\cos s}s\, ds,\quad t\in\mathbb{R}.
\end{equation*}
is the cosine integral function. \par
This expression of $\Theta _{\sigma }$ shows that this kernel behaves badly
in general: it does not even belong to $L_{loc}^{\infty }$ (see Corollary %
\ref{Thetalinfty}) and has no decay at infinity (see Corollary \ref{bo2}).
In spite of these facts, it was proved in \cite{cgn2} that some directional smoothing effect is still present, but the analysis carried on there also shows the needed of a systematic and general study of the boundedness of
such operators $\limfunc{Op}\nolimits_{\mathrm{BJ}}(a)$ on modulation spaces, in dependence of the Born-Jordan symbol space. Modulation spaces, introduced by Feichtinger in \cite{feichtinger83}, have been widely employed in the literature to investigate properties of pseudodifferential operators, in particular we highlight the contributions \cite{benyi2,A1,CTW2014,GH99,labate2,Sjostrand1,sjostrand95,sugitomita2,A7,toft1,Toftweight}.    For their definition and main properties  we refer  to the successive section.

The main result concerning the sufficient boundedness conditions of
Born-Jordan operators on modulation spaces shows that they behave similarly
to Weyl pseudodifferential operators or any other $\tau $-form of
pseudodifferential operators. For comparison, see \cite[Theorem 5.2,
Proposition 5.3]{cordero2}, \cite[Theorem 1.1]{cordero3} and \cite[Theorem
4.3]{toft1}. The necessary boundedness conditions still contain some open
problems, as shown in the following result. We denote $q^{\prime }$ the
conjugate exponent of $q\in \lbrack 1,\infty ]$; it is defined by $%
1/q+1/q^{\prime }=1$.

\begin{theorem}
\label{Charpseudo} Consider $1\leq p,q,r_1,r_2\leq \infty$, such that 
\begin{equation}  \label{indicitutti}
p\leq q^{\prime }
\end{equation}
\noindent and 
\begin{equation}  \label{indiceq}
\quad q \leq \min \{r_1,r_2,r_1^{\prime },r_2^{\prime }\}.
\end{equation}
Then the Born-Jordan operator $\limfunc{Op}\nolimits_{\mathrm{BJ}}(a)$, from 
$\mathcal{S}(\mathbb{R}^d)$ to $\mathcal{S}^{\prime }(\mathbb{R}^d)$, having
symbol $a \in M^{p,q}(\mathbb{R}^{2d})$, extends uniquely to a bounded
operator on $\mathcal{M}^{r_1,r_2}(\mathbb{R}^d)$, with the estimate 
\begin{equation}  \label{stimaA}
\|\limfunc{Op}\nolimits_{\mathrm{BJ}}(a) f\|_{\mathcal{M}^{r_1,r_2}}
\lesssim \|a\|_{M^{p,q}}\|f\|_{\mathcal{M}^{r_1,r_2}}\quad f\in \mathcal{M}%
^{r_1,r_2}.
\end{equation}
Vice-versa, if this conclusion holds true, the constraints %
\eqref{indicitutti} is satisfied and it must hold 
\begin{equation}  \label{stimanew}
\max \left\{\frac 1{r_1},\frac 1{r_2},\frac 1{r_1^{\prime }},\frac
1{r_2^{\prime }}\right\}\leq \frac 1q+\frac1p
\end{equation}
which is \eqref{indiceq} when $p=\infty$.
\end{theorem}

Notice that the condition \eqref{stimanew} is weaker than \eqref{indiceq}
when $p<\infty$. The condition \eqref{stimanew} is obtained by working with
rescaled Gaussians which provide the best localization in terms of Wigner
distribution (cf. \cite{Lieb}). On the Fourier side, the Born-Jordan distribution is the
point-wise multiplication of the Wigner distribution with the kernel $\Theta$%
. This reasoning conduces to conjecture that the condition \eqref{stimanew}
should be the optimal one so that the sufficient boundedness conditions for
Born-Jordan operators might be weaker than the corresponding ones for Weyl
and $\tau$-pseudodifferential operators. But the matter is really subtle and
requires a new and most refined analysis of the kernel $\Theta$. In particular the zeroes of the $\Theta $ function should play a key for a thorough understanding of such operators, which certainly deserve further study.\par\medskip
The paper is organized as follows. Section 2 is devoted to some preliminary results from Time-frequency Analysis. In Section 3 we perform an analysis of the kernel $\Theta$ and we prove the above formula for $\Theta_{\sigma}$. In Sections 4 and 5 we study the Cohen kernels and the boundedness of Born-Jordan operators in the framework of modulation spaces.


\section{Preliminaries}

In this section we recall the definition of the spaces involved in our study
and present the main time-frequency tools used.

\medskip \noindent \textbf{Modulation and Wiener amalgam spaces.} The
modulation and Wiener amalgam space norms are a measure of the joint
time-frequency distribution of $f\in \mathcal{S} ^{\prime }$. For their
basic properties we refer to the original literature \cite%
{feichtinger80,feichtinger83,feichtinger90} and the textbooks \cite%
{Birkbis,book}.

Let $f\in\mathcal{S}^{\prime }(\mathbb{R}^d)$. We define the short-time
Fourier transform of $f$ as 
\begin{equation}  \label{STFTdef}
V_gf(z)=\langle f,\pi(z)g\rangle=\mathcal{F} [fT_x g](\omega)=\int_{\mathbb{R%
}^d} f(y)\, {\overline {g(y-x)}} \, e^{-2\pi iy \omega }\,dy
\end{equation}
for $z=(x,\omega)\in\mathbb{R}^d\times\mathbb{R}^d$.


Given a non-zero window $g\in\mathcal{S}(\mathbb{R}^d)$, $1\leq p,q\leq
\infty$, the \textit{\ modulation space} $M^{p,q}(\mathbb{R}^d)$ consists of
all tempered distributions $f\in\mathcal{S}^{\prime }(\mathbb{R}^d)$ such
that $V_gf\in L^{p,q}(\mathbb{R}^{2d} )$ (weighted mixed-norm spaces). The
norm on $M^{p,q}$ is 
\begin{equation*}
\|f\|_{M^{p,q}}=\|V_gf\|_{L^{p,q}}=\left(\int_{\mathbb{R}^d} \left(\int_{%
\mathbb{R}^d}|V_gf(x,\omega)|^p(x,\omega)^p\,
dx\right)^{q/p}d\omega\right)^{1/p} \,
\end{equation*}
(with natural modifications when $p=\infty$ or $q=\infty$). If $p=q$, we
write $M^p$ instead of $M^{p,p}$.

The space $M^{p,q}(\mathbb{R}^d)$ is a Banach space whose definition is
independent of the choice of the window $g$, in the sense that different
nonzero window functions yield equivalent norms. The modulation space $%
M^{\infty,1}$ is also called Sj\"ostrand's class \cite{Sjostrand1}.

The closure of $\mathcal{S}(\mathbb{R}^d)$ in the $M^{p,q}$-norm is denoted $%
\mathcal{M}^{p,q}(\mathbb{R}^d)$. Then 
\begin{equation*}
\mathcal{M}^{p,q}(\mathbb{R}^d) \subseteq M^{p,q}(\mathbb{R}^d), \quad {%
\mbox and }\,\, \mathcal{M}^{p,q} (\mathbb{R}^d) = M^{p,q} (\mathbb{R}^d),
\end{equation*}
provided $p<\infty$ and $q<\infty$.

Recalling that the conjugate exponent $p^{\prime }$ of $p\in \lbrack
1,\infty ]$ is defined by $1/p+1/p^{\prime }=1$, for any $p,q\in \lbrack
1,\infty ]$ the inner product $\langle \cdot ,\cdot \rangle $ on $\mathcal{S}%
(\mathbb{R}^{d})\times \mathcal{S}(\mathbb{R}^{d})$ extends to a continuous
sesquilinear map $M^{p,q}(\mathbb{R}^{d})\times M^{p^{\prime },q^{\prime }}(%
\mathbb{R}^{d})\rightarrow \mathbb{C}$.

Modulation spaces enjoy the following inclusion properties: 
\begin{equation}
\mathcal{S}(\mathbb{R}^{d})\subseteq M^{p_{1},q_{1}}(\mathbb{R}%
^{d})\subseteq M^{p_{2},q_{2}}(\mathbb{R}^{d})\subseteq \mathcal{S}^{\prime
}(\mathbb{R}^{d}),\quad p_{1}\leq p_{2},\,\,q_{1}\leq q_{2}.
\label{modspaceincl1}
\end{equation}%
The Wiener amalgam spaces $W(\mathcal{F}L^{p},L^{q})(\mathbb{R}^{d})$ are
given by the distributions $f\in \mathcal{S}^{\prime }(\mathbb{R}^{d})$ such
that 
\begin{equation*}
\Vert f\Vert _{W(\mathcal{F}L^{p},L^{q})(\mathbb{R}^{d})}:=\left( \int_{%
\mathbb{R}^{d}}\left( \int_{\mathbb{R}^{d}}|V_{g}f(x,\omega )|^{p}\,d\omega
\right) ^{q/p}dx\right) ^{1/q}<\infty \,
\end{equation*}%
(with obvious changes for $p=\infty $ or $q=\infty $). Using Parseval
identity in \eqref{STFTdef}, we can write the so-called fundamental identity
of time-frequency analysis\thinspace\ $V_{g}f(x,\omega )=e^{-2\pi ix\omega
}V_{\hat{g}}\hat{f}(\omega ,-x)$, hence 
\begin{equation*}
|V_{g}f(x,\omega )|=|V_{\hat{g}}\hat{f}(\omega ,-x)|=|\mathcal{F}(\hat{f}%
\,T_{\omega }\overline{\hat{g}})(-x)|
\end{equation*}%
so that 
\begin{equation*}
\Vert f\Vert _{{M}^{p,q}}=\left( \int_{\mathbb{R}^{d}}\Vert \hat{f}\
T_{\omega }\overline{\hat{g}}\Vert _{\mathcal{F}L^{p}}^{q}(\omega )\ d\omega
\right) ^{1/q}=\Vert \hat{f}\Vert _{W(\mathcal{F}L^{p},L^{q})}.
\end{equation*}
This means that these Wiener amalgam spaces are simply the image under
Fourier transform\thinspace\ of modulation spaces: 
\begin{equation}
\mathcal{F}({M}^{p,q})=W(\mathcal{F}L^{p},L^{q}).  \label{W-M}
\end{equation}
We will often use the following product property of Wiener amalgam spaces (%
\cite[Theorem 1 (v)]{feichtinger80}, Theorem 1 (v)): 
\begin{equation}
f\in W(\mathcal{F}L^{1},L^{\infty })\mbox{ and } g\in W(\mathcal{F}
L^{p},L^{q})\Longrightarrow fg\in W(\mathcal{F}L^{p},L^{q}).  \label{product}
\end{equation}
In order to prove the necessary boundedness conditions for Born-Jordan
operators we shall use the dilation properties for Gaussian functions.
Precisely, consider $\varphi (x)=e^{-\pi |x|^{2}}$ and define 
\begin{equation}
\varphi _{\lambda }(x)=\varphi (\sqrt{\lambda }x)=e^{-\pi \lambda
|x|^{2}},\quad \lambda >0.  \label{flam}
\end{equation}%
The dilation properties for the Gaussian $\varphi _{\lambda }$ in modulation
spaces were proved in \cite[Lemma 1.8]{toft1} (see also \cite[Lemma 3.2]%
{CNJFA2008}).

\begin{lemma}
\label{lemma5.2-zero} For $1\leq p,q\leq \infty$, we have 
\begin{equation}  \label{eqa0-zero}
\|\varphi_{\lambda}\|_{M^{p,q}}\asymp \lambda^{-\frac d{2q^{\prime }}}\quad 
\mathrm{as}\ \lambda\to+\infty
\end{equation}
\begin{equation}  \label{eqa1-zero}
\|\varphi_{\lambda}\|_{M^{p,q}}\asymp \lambda^{-\frac d{2p}}\quad \mathrm{as}\
\lambda\to0^+.
\end{equation}
\end{lemma}

The following dilation properties are a straightforward generalization of \cite[Lemma 2.3]{cgn2}.

\begin{lemma}
\label{lemma5.2} Consider $1\leq p,q\leq \infty$, $\psi\in C^\infty_c(%
\mathbb{R}^d)\setminus\{0\}$ and $\lambda>0$. Then 
\begin{equation}  \label{eqa0}
\|\psi(\sqrt{\lambda}\,\cdot)\|_{W(\mathcal{F} L^p,L^q)}\asymp
\lambda^{-\frac d{2p^{\prime }}}\quad \mathrm{as}\ \lambda\to+\infty
\end{equation}
\begin{equation}  \label{eqa1}
\|\psi(\sqrt{\lambda}\,\cdot)\|_{W(\mathcal{F} L^p,L^q)}\asymp
\lambda^{-\frac d{2q}}\quad \mathrm{as}\ \lambda\to0^+.
\end{equation}
The same conclusion holds uniformly with respect to $\lambda$ if $\psi$ varies in bounded subsets of $C^\infty_c(\mathbb{R}^d)$.
\end{lemma}

Another tool for obtaining the optimality of our results is the cross-Wigner
distribution of rescaled Gaussian functions. The proof is a straightforward
computation (see Prop. 244 in \cite{Birkbis}):

\begin{lemma}
\label{l1} Consider $\varphi (x)=e^{-\pi |x|^{2}}$ and $\varphi _{\lambda }$
as in \eqref{flam}. Then 
\begin{equation}
W(\varphi ,\varphi _{\lambda })(x,\omega )=\frac{2^{d}}{(\lambda +1)^{\frac{d%
}{2}}}e^{-\frac{4\pi \lambda }{\lambda +1}|x|^{2}}e^{-\frac{4\pi }{\lambda +1%
}|\omega |^{2}}e^{-4\pi i\frac{\lambda -1}{\lambda +1}x\omega }.
\label{wign1l}
\end{equation}
\end{lemma}

It follows that:

\begin{cor}
Consider $\varphi $ and $\varphi _{\lambda }$ as in the assumptions of Lemma %
\ref{l1}. Then 
\begin{equation}
\mathcal{F}W(\varphi ,\varphi _{\lambda })(\zeta _{1},\zeta _{2})=\frac{1}{%
(\lambda +1)^{\frac{d}{2}}}e^{-\frac{\pi  }{\lambda +1}\zeta _{1}^{2}}e^{-%
\frac{\pi \lambda }{\lambda +1}\zeta _{2}^{2}}e^{-\pi i\frac{\lambda -1}{\lambda +1}%
\zeta _{1}\zeta _{2}}.  \label{cfwffl}
\end{equation}
\end{cor}
\begin{proof}
Formula \eqref{cfwffl} is easily obtained from  \eqref{wign1l} using
well-known Gaussian integral formulas; it can also be painlessly obtained
from \eqref{wign1l} by observing that for any functions $\psi ,\phi \in
L^{2}(\mathbb{R}^{d})$ the following relation between the cross-Wigner
distribution and its Fourier transform holds: 
\begin{equation*}
\mathcal{F}W(\psi ,\phi )(x,\o)=2^{-d}W(\psi ,\phi ^{\vee })(\tfrac{1}{2}\o,-%
\tfrac{1}{2}x)
\end{equation*}%
where $\phi ^{\vee }(x)=\phi (-x)$ (see formula (9.27) in \cite{Birkbis}, or
formula (1.90) in Folland \cite{folland}).
\end{proof}



We denote by $\sigma$ the symplectic form on the phase space ${\mathbb{R}%
^{2d}}\equiv\mathbb{R}^d\times \mathbb{R}^d$; the phase space variable is
denoted $z=(x,\omega)$ and the dual variable by $\zeta=(\zeta_1,\zeta_2)$.
By definition $\sigma(z,\zeta)=Jz\cdot \zeta=\omega\cdot \zeta_1-x\cdot
\zeta_2$, where 
\begin{equation}  \label{J}
J=%
\begin{pmatrix}
0_{d\times d} & I_{d\times d} \\ 
-I_{d\times d} & 0_{d\times d}%
\end{pmatrix}%
.
\end{equation}

The Fourier transform of a function $f$ on $\mathbb{R}^{d}$ is normalized as 
\begin{equation*}
\mathcal{F}f(\omega )=\int_{\mathbb{R}^{d}}e^{-2\pi ix\omega }f(x)\,dx,
\end{equation*}%
and the symplectic Fourier transform of a function $F$ on the phase space ${%
\mathbb{R}^{2d}}$ is 
\begin{equation}
\mathcal{F}_{\sigma }F(\zeta )=\int_{{\mathbb{R}^{2d}}}e^{-2\pi {i}\sigma
(\zeta ,z)}F(z)\,dz.  \label{SFT}
\end{equation}%
Observe that $\mathcal{F}_{\sigma }F(\zeta )=\mathcal{F}F(J\zeta )$. Hence
the symplectic Fourier transform of the Wigner distribution \eqref{wign1l}
is given by 
\begin{equation}
\mathcal{F}_{\sigma }W(\varphi ,\varphi _{\lambda })(\zeta _{1},\zeta _{2})=%
\frac{1}{(\lambda +1)^{\frac{d}{2}}}e^{-\frac{\pi \lambda}{\lambda +1}\zeta
_{1}^{2}}e^{-\frac{\pi }{\lambda +1}\zeta _{2}^{2}}e^{\pi i\frac{\lambda -1}{%
\lambda +1}\zeta _{1}\zeta _{2}}.  \label{cfwfflsig}
\end{equation}

We will also use the important relation 
\begin{equation}  \label{eq10}
\mathcal{F}_{\sigma}[F\ast G]=\mathcal{F}_{\sigma} F\, \mathcal{F}_{\sigma} G.
\end{equation}

The convolution relations for modulation spaces are essential in the proof
of the boundedness of a Born-Jordan operator on these spaces and were proved
in \cite[Proposition 2.4]{CG02}:

\begin{proposition}
\label{mconvmp} Let $1\leq p,q,r,s,t\leq\infty$. If 
\begin{equation*}
\frac1p+\frac1q-1=\frac1r,\quad \,\, \, \quad\frac1t+\frac1{t^{\prime }}=1\,
,
\end{equation*}
then 
\begin{equation}  \label{mconvm}
M^{p,st}(\mathbb{R}^d)\ast M^{q,st^{\prime }}(\mathbb{R}^d)\hookrightarrow
M^{r,s}(\mathbb{R}^d)
\end{equation}
with $\| f\ast h \|_{M^{r,s}}\lesssim \|f\|_{M^{p,st}}\|h\|_{
M^{q,st^{\prime }}}$.
\end{proposition}

We also recall the useful result proved in \cite[Lemma 5.1]{cgn2}.

\begin{lemma}
\label{lemma5.1} Let $\chi\in C^\infty_c(\mathbb{R})$. Then, for $%
\zeta_1,\zeta_2\in\mathbb{R}^d$, the function $\chi(\zeta_1 \zeta_2)$
belongs to $W(\mathcal{F} L^1,L^\infty)({\mathbb{R}^{2d}})$.
\end{lemma}

\section{Analysis of the Cohen kernel $\Theta$}

Consider the Cohen kernel $\Theta$ defined in \eqref{sincxp}. Obviously $%
\Theta\in\mathcal{C}^\infty({\mathbb{R}^{2d}})\cap L^\infty({\mathbb{R}^{2d}}%
)$ but displays a vary bad decay at infinity, as clarified in what follows.

\begin{proposition}
\label{thetalp} For $1\leq p< \infty$, the function $\Theta\notin L^p({%
\mathbb{R}^{2d}})$.
\end{proposition}

\begin{proof}
Observe that, for $t\in\mathbb{R}$, $|t|\leq 1/2$, the function $\mathrm{sinc%
}t$ satisfies $|\mathrm{sinc}\,t|\geq 2/\pi$. Hence, for any $1\leq p<\infty$%
, 
\begin{align*}
\int_{{\mathbb{R}^{2d}}}|\Theta(x,\omega)|^p\,dx d\omega&= \int_{{\mathbb{R}
^{2d}}} |\mathrm{sinc}(x\omega)|^p\,dx d\omega \\
&\geq \int_{|xp|\leq 1/2} |\mathrm{sinc}(x\omega)|^p\,dx d\omega \\
&\geq \left(\frac 2\pi\right)^p \int_{|x\omega|\leq 1/2} \,d x d\omega \\
& =\left(\frac 2\pi\right)^p meas\{(x,\omega) \,:\,|x\omega|\leq
1/2\}=+\infty.
\end{align*}
This concludes the proof.
\end{proof}

We continue our investigation of the function $\Theta$ by looking for the
right Wiener amalgam and modulation spaces containing this function. For
this reason, we first reckon explicitly the STFT of the $\Theta$ function,
with respect to the Gaussian window $g(x,\omega)=e^{-\pi x^2} e^{-\pi
\omega^2}\in\mathcal{S}({\mathbb{R}^{2d}})$.

\begin{proposition}
\label{sincSTFT} For $z_1,z_2, \zeta_1,\zeta_2\in\mathbb{R}^d$, 
\begin{align}  
V_g & \Theta (z_1,z_2, \zeta_1,\zeta_2) \\
&=\int_{-1/2}^{1/2} \frac{1}{(t^2+1)^{d/2}} e^{-2\pi i [\frac1t \zeta_1
\zeta_2+\frac{t}{t^2+1} (z_1-\frac1t \zeta_2) (z_2-\frac1t \zeta_1)]}e^{-\pi 
\frac{t^2}{t^2+1} [(z_1-\frac1t \zeta_2)^2+(z_2-\frac1t \zeta_1)^2]} \,dt.  \label{Vgsinc}
\end{align}
\end{proposition}

\begin{proof}
We write $\Theta(z_1,z_2)=F_1(z_1,z_2)+F_2(z_1,z_2)$, where $F_1(z_1,z_2)=\int_0^{1/2} e^{2\pi i z_1 z_2 t} dt$ and $F_2(z)=F_1(Jz)$, $z=(z_1,z_2)$. Let us first reckon $V_g F_1(z,\zeta)$, $z=(z_1,z_2)$, $\zeta=(\zeta_1,\zeta_2)\in{\mathbb{R}^{2d}}$, where $g$ is the Gaussian
function above. For $t>0$ we define the function $H_t(z_1,z_2)=e^{2\pi i t
z_1 z_2}$ and observe that 
\begin{equation}  \label{ftchirp}
\mathcal{F} H_t (\zeta_1,\zeta_2)=\frac 1{t^{d}} e^{-2\pi i\frac1t
\zeta_1\zeta_2 }
\end{equation}
(cf. \cite[Appendix A, Theorem 2]{folland}). By the Dominated Convergence
Theorem, 
\begin{align*}
V_g F_1(z,\zeta)&=\int_0^{1/2}\mathcal{F} (H_tT_z g)(\zeta) dt=\int_0^{1/2}(\mathcal{F}(H_t)\ast M_{-z}\hat{g})(\zeta_1,\zeta_2) dt \\
&=\int_0^{1/2} \frac1{t^d} \int_{{\mathbb{R}^{2d}}} e^{-2\pi i\frac1 t
(\zeta_1- y_1)\cdot(\zeta_2-y_2)} e^{-2\pi i (z_1,z_2)\cdot (y_1,y_2)}
e^{-\pi y_1^2} e^{-\pi y_2^2} \,dy_1 dy_2 dt \\
&=\int_0^{1/2} \frac1{t^d} e^{-2\pi i \frac1t \zeta_1 \zeta_2} \!\!\int_{{%
\mathbb{R}^{2d}}} \!\! e^{-2\pi i\frac1t y_1 y_2 +2\pi i \frac1t (\zeta_2
y_1+\zeta_1 y_2)-2\pi i (z_1 y_1 + z_2 y_2)} e^{-\pi (y_1^2+ y_2^2)}\, dy_1
dy_2dt \\
&= \int_0^{1/2} \frac1{t^d} e^{-2\pi i \frac1t \zeta_1 \zeta_2}\int_{\mathbb{%
R}^d} e^{2\pi i ( \frac1t \zeta_1 y_2-z_2 y_2)} e^{-\pi y_2^2} \\
&\quad\quad \quad \cdot\quad \left( \int_{\mathbb{R}^d} e^{-2\pi i y_1\cdot
( \frac1t y_2- \frac1t \zeta_2 +z_1)} e^{-\pi y_1^2}\, dy_1\right) \, dy_2 dt
\\
&=\int_0^{1/2} \frac1{t^d} e^{-2\pi i \frac1t \zeta_1 \zeta_2} \int_{\mathbb{R}^d} e^{-2\pi i y_2\cdot (z_2- \frac1t \zeta_1)} e^{-\pi y_2^2} e^{-\pi(
\frac1t y_2- \frac1t \zeta_2+z_1)^2}\,dy_2 dt \\
&=\int_0^{1/2} \!\!\frac1{t^d} e^{-2\pi i \frac1t \zeta_1 \zeta_2} e^{-\pi
(z_1-\frac1t \zeta_2)^2}\!\int_{\mathbb{R}^d} \!e^{-2\pi i y_2\cdot
(z_2-\frac1t \zeta_1)} e^{-\pi ((1+\frac1{t^2})y_2^2-2(z_1-\frac1t
\zeta_2)\cdot \frac1t y_2)} dy_2 dt \\
&= \int_0^{1/2} \!\frac1{t^d} e^{-2\pi i \frac1t \zeta_1 \zeta_2} e^{-\pi 
\frac{t^2}{t^2+1} (u_1-\frac1t \zeta_2)^2}\int_{\mathbb{R}^d} e^{-2\pi i y_2
\cdot (u_2-\frac1t \zeta_1)} \\
&\quad\quad\quad \cdot\quad e^{-\pi (\frac{\sqrt{t^2+1}}{t}y_2-\frac{t}{%
\sqrt{t^2+1}}(z_1-\frac1t \zeta_2))^2}dy_2 dt \\
&=\int_0^{1/2}\frac{1}{(t^2+1)^{d/2}}e^{-2\pi i\frac 1t \zeta_1 \zeta_2}
e^{-\pi \frac{t^2}{t^2+1} (z_1-\frac1t \zeta_2)^2} \\
&\quad\quad\quad \cdot\quad \int_{\mathbb{R}^d} e^{-2\pi i(\frac{t}{\sqrt{t^2+1}}w +\frac{t}{t^2+1} (z_1-\frac1t \zeta_2))(z_2-\frac1t
\zeta_1)}e^{-\pi w^2}dw dt \\
&=\int_0^{1/2} \!\!\!\frac{1}{(t^2+1)^{d/2}} e^{-2\pi i \frac1t \zeta_1
\zeta_2}e^{-\pi \frac{t^2}{t^2+1} (z_1-\frac1t \zeta_2)^2}e^{-\pi \frac{t^2}{t^2+1} (z_2-\frac1t \zeta_1)^2} \\
&\quad\quad\quad \cdot\quad e^{-2\pi i \frac{t}{t^2+1} (z_1-\frac1t \zeta_2)
(z_2-\frac1t \zeta_1)}dt.
\end{align*}
Now, an easy computation shows 
\begin{equation*}
V_g F_2 (z,\zeta)= V_g F_1(Jz,J\zeta)
\end{equation*}
so that $V_g \Theta=V_g F_1+V_g F_2$ and we obtain \eqref{Vgsinc}.
\end{proof}

\begin{proposition}
\label{sincxpwiener} The function $\Theta$ in \eqref{sincxp} belongs to $W(\mathcal{F} L^1,L^\infty)({\mathbb{R}^{2d}})$.
\end{proposition}

\begin{proof}
We simply have to calculate 
\begin{equation*}
\sup_{z\in{\mathbb{R}^{2d}}}\int_{{\mathbb{R}^{2d}}} |V_g \Theta(z,\zeta)|
d\zeta.
\end{equation*}
From \eqref{Vgsinc} we observe that 
\begin{align*}
\|V_g \Theta (z,\cdot)\|_1&\leq \int_{-1/2}^{1/2} \int_{{\mathbb{R}^{2d}}} 
\frac{1}{(t^2+1)^{d/2}} e^{-\pi \frac{t^2}{t^2+1} (z_1-\frac1t
\zeta_2)^2}e^{-\pi \frac{t^2}{t^2+1} (z_2-\frac1t \zeta_1)^2} d\zeta_1
d\zeta_2 dt \\
&= \int_{-1/2}^{1/2} \int_{{\mathbb{R}^{2d}}} \frac{1}{(t^2+1)^{d/2}}
e^{-\pi \frac{1}{t^2+1} ( t z_1- \zeta_2)^2}e^{-\pi \frac{1}{t^2+1} (t z_2-
\zeta_1)^2} d\zeta_1 d\zeta_2 dt \\
&= \int_{-1/2}^{1/2} \int_{{\mathbb{R}^{2d}}} (t^2+1)^{d/2} e^{-\pi ( v_1^2
+v_2^2)} dv_1 dv_2 dt=C<\infty,
\end{align*}
from which the claim follows.
\end{proof}

Using the STFT of the function $\Theta$ in \eqref{Vgsinc} we observe that 
\begin{equation*}
\|V_g \Theta (\cdot,\zeta)\|_1\leq \int_{-1/2}^{1/2} \int_{{\mathbb{R}^{2d}}%
} \frac{1}{(t^2+1)^{d/2}} e^{-\pi \frac{t^2}{t^2+1} (u_1-\frac1t
\zeta_2)^2}e^{-\pi \frac{t^2}{t^2+1} (u_2-\frac1t \zeta_1)^2} du_1 du_2
dt=+\infty
\end{equation*}
so that we conjecture that $\Theta \notin M^{1,\infty}({\mathbb{R}^{2d}})$.
The previous claim will follow if we prove that $\Theta_{\sigma} \notin W(%
\mathcal{F} L^1,L^\infty)({\mathbb{R}^{2d}})$.

Note that $\Theta_{\sigma}(\zeta)=\mathcal{F} \Theta(J\zeta)=\mathcal{F}
\Theta(\zeta)$. Furthermore, the distributional Fourier transform of $\Theta$
can be computed explicitly as follows. First, recall the definition of the
cosine integral function \eqref{cosint}.

\begin{proposition}
\label{Thetasigma} For $d\geq 1$ the distribution symplectic Fourier
transform \thinspace\ $\Theta _{\sigma }$ of the function $\Theta $ is
provided by 
\begin{equation}
\Theta _{\sigma }(\zeta _{1},\zeta _{2})=%
\begin{cases}
-2\,\mathrm{Ci}(4\pi |\zeta _{1}\zeta _{2}|),\quad (\zeta _{1},\zeta
_{2})\in \mathbb{R}^{2},\,d=1 \\ 
\mathcal{F}(\chi _{\{|s|\geq 2\}}|s|^{d-2})(\zeta _{1},\zeta _{2}),\quad
(\zeta _{1}\zeta _{2})\in {\mathbb{R}^{2d}},\,d\geq 2,%
\end{cases}
\label{C1intro}
\end{equation}%
where $\chi _{\{|s|\geq 2\}}$ is the characteristic function of the set $%
\{s\in \mathbb{R}\,:\,|s|\geq 2\}$. The case $d=1$ can be recaptured by the
case $d\geq 2$ using \eqref{cosint}.
\end{proposition}

\begin{proof}
We carry out the computations of $\Theta_{\sigma}$ by studying first the case
in dimension $d=1$ and secondly, inspired by the former case, $d>1$.

\emph{First step: $d=1$}. By Proposition \ref{thetalp}, the function $\Theta$
is in 
\begin{equation*}
L^\infty(\mathbb{R}^2)\setminus L^p(\mathbb{R}^2)\subset \mathcal{S}^{\prime
}(\mathbb{R}^2), \quad 1\leq p<\infty,
\end{equation*}
so that the Fourier transform is meant in $\mathcal{S}^{\prime }(\mathbb{R}%
^2)$. Observe that 
\begin{equation*}
\mathcal{F} \Theta (\zeta_1,\zeta_2)=\mathcal{F}_2 \mathcal{F}%
_1\Theta(\zeta_1,\zeta_2),
\end{equation*}
where $\mathcal{F}_1$ (resp. $\mathcal{F}_2$) is the partial Fourier
transform\, with respect to the first (resp. second) variable. Indeed, for
every test function $\varphi\in\mathcal{S}(\mathbb{R}^2)$, 
\begin{equation*}
\langle \mathcal{F}\Theta,\varphi\rangle=\langle \Theta, \mathcal{F}%
^{-1}\varphi \rangle
\end{equation*}
and $\mathcal{F}^{-1}\varphi(x,\omega)=\mathcal{F}^{-1}_1\mathcal{F}%
^{-1}_2\varphi(x,\omega)=\mathcal{F}^{-1}_2\mathcal{F}^{-1}_1\varphi(x,%
\omega)$, by Fubini's Theorem.

Using 
\begin{equation*}
\mathcal{F}_{1}{\mathrm{sinc}(y_{2}\cdot )}(\zeta _{1})=\frac{1}{|y_{2}|}%
p_{1/2}(\zeta _{1}/y_{2}),\quad y_{2}\not=0,
\end{equation*}%
where $p_{1/2}(t)$ is the box function defined by $p_{1/2}(t)=1$ for $%
|t|\leq 1/2$ and $p_{1/2}(t)=0$ otherwise, we obtain, for $\zeta _{1}\zeta
_{2}\not=0$ (hence in particular $|\zeta _{1}|>0$), 
\begin{align*}
\mathcal{F}\Theta (\zeta _{1},\zeta _{2})& =\int_{\mathbb{R}}e^{-2\pi i\zeta
_{2}y_{2}}\frac{1}{|y_{2}|}p_{1/2}(\zeta
_{1}/y_{2})\,dy_{2}=\int_{|y_{2}|\geq 2|\zeta _{1}|}e^{-2\pi i\zeta
_{2}y_{2}}\frac{1}{|y_{2}|}dy_{2} \\
& =\int_{|s|\geq 2|\zeta _{1}\zeta _{2}|}e^{-2\pi is}\frac{1}{|s|}\,ds \\
& =\int_{|s|\geq 2|\zeta _{1}\zeta _{2}|}\frac{\cos (2\pi s)-i\sin (2\pi s)}{%
|s|}\,ds \\
& =\int_{|s|\geq 2|\zeta _{1}\zeta _{2}|}\frac{\cos {2\pi s}}{|s|}\,ds \\
& =2\int_{2|\zeta _{1}\zeta _{2}|}^{+\infty }\frac{\cos {2\pi s}}{s}\,ds=-2%
\mathrm{Ci}(4\pi |\zeta _{1}\zeta _{2}|),
\end{align*}%
by \eqref{cosint}, so that, since $\zeta _{1}\zeta _{2}=0$ is a set of
Lebesgue measure equal zero on $\mathbb{R}^{2}$, we can write 
\begin{equation}
\Theta _{\sigma }(\zeta _{1},\zeta _{2})=-2\mathrm{Ci}(4\pi |\zeta _{1}\zeta
_{2}|),\quad (\zeta _{1},\zeta _{2})\in \mathbb{R}^{2}.  \label{Ci}
\end{equation}%
\smallskip \noindent \emph{Second step: $d>1$}. This is a simple
generalization on the former step. For $(z_{1},z_{2}),(\zeta _{1},\zeta
_{2})\in {\mathbb{R}^{2d}}$, $d>1$, we write 
\begin{equation}
z_{i}=(z_{i}^{\prime },z_{i,d}),\,\zeta _{i}=(\zeta _{i}^{\prime },\zeta
_{i,d}),\quad z_{i}^{\prime },\zeta _{i}^{\prime }\in \mathbb{R}%
^{d-1},\,\,z_{i,d},\zeta _{i,d}\in \mathbb{R},\quad i=1,2.  \label{variables}
\end{equation}%
We decompose $\mathcal{F}\Theta =\mathcal{F}_{2d}\mathcal{F}^{\prime }%
\mathcal{F}_{1}\Theta $ where, for $\Theta =\Theta (z_{1},z_{2})$, $\mathcal{%
F}_{1}$ is the partial Fourier transform\thinspace\ with respect to the
variable $z_{1,d}$, $\mathcal{F}^{\prime }$ is the partial Fourier transform
\thinspace\ with respect the $2d-2$ variables $(z_{1}^{\prime
},z_{2}^{\prime })\in \mathbb{R}^{2d-2}$ and $\mathcal{F}_{2d}$ is the
partial Fourier transform  with respect to the last variable $z_{2,d}$. We start with computing the partial Fourier transform $\mathcal{F}_{1}$: 
\begin{align*}
\mathcal{F}_{1}\Theta (z_{1}^{\prime },\cdot ,z_{2}^{\prime },z_{2,d})(\zeta_{1,d})& =\mathcal{F}_{1}(T_{\frac{-z_{1}^{\prime }z_{2}^{\prime }}{z_{2,d}}}
\mbox{sinc}(z_{2,d}\cdot ))(\zeta_{1,d}) \\
& =e^{2\pi i\frac{\zeta_{1,d}}{z_{2,d}}z_{1}^{\prime}z_{2}^{\prime}}\frac{1}{|z_{2,d}|}\mathcal{F}_{1}(\mbox{sinc})\left(\frac{\zeta_{1,d}}{z_{2,d}}\right)  \\
& =e^{2\pi i\frac{\zeta_{1,d}}{z_{2,d}}z_{1}^{\prime }z_{2}^{\prime }}\frac{1}{|z_{2,d}|}p_{1/2}\left(\frac{\zeta_{1,d}}{z_{2,d}}\right).
\end{align*}
Using the Gaussian integrals in \cite[Appendix A, Theorem 2]{folland}) we
calculate 
\begin{equation*}
\mathcal{F}^\prime (e^{2\pi i\frac{\zeta _{1,d}}{z_{2,d}}z_{1}^{\prime
}z_{2}^{\prime }})(\zeta _{1}^{\prime },\zeta _{2}^{\prime })=\left\vert 
\frac{z_{2,d}}{\zeta _{1,d}}\right\vert ^{d-1}e^{-2\pi i\frac{z_{2,d}}{\zeta
_{1,d}}\zeta _{1}^{\prime }\zeta _{2}^{\prime }},
\end{equation*}%
so that 
\begin{align*}
\mathcal{F}\Theta (\zeta _{1},\zeta _{2})& =\mathcal{F}_{2d}\left( e^{-2\pi i%
\frac{z_{2,d}}{\zeta _{1,d}}\zeta _{1}^{\prime }\zeta _{2}^{\prime
}}\left\vert \frac{z_{2,d}}{\zeta _{1,d}}\right\vert ^{d-1}\frac{1}{|z_{2,d}|%
}p_{1/2}\left( \frac{\zeta _{1,d}}{z_{2,d}}\right) \right) (\zeta _{2,d}) \\
& =\int_{\left\vert \frac{\zeta _{1,d}}{z_{2,d}}\right\vert \leq \frac{1}{2}%
}\left\vert \frac{z_{2,d}}{\zeta _{1,d}}\right\vert ^{d-1}\frac{1}{|z_{2,d}|}%
e^{-2\pi i\frac{z_{2,d}}{\zeta _{1,d}}\zeta _{1}\zeta _{2}}\,dz_{2,d} \\
& =\int_{|s|\geq 2}e^{-2\pi is(\zeta _{1}\zeta _{2})}|s|^{d-2}\,ds,
\end{align*}%
as claimed.
\end{proof}

Notice that the second equation \eqref{C1intro} can be written%
\begin{equation*}
\Theta _{\sigma }(\zeta _{1},\zeta _{2})=\int_{|s|\geq 2}e^{-2\pi is(\zeta
_{1}\zeta _{2})}|s|^{d-2}\,ds.
\end{equation*}

\begin{corollary}
\label{Thetalinfty} We have 
\begin{equation*}
\Theta _{\sigma }\notin L_{loc}^{\infty }({\mathbb{R}^{2d}}).
\end{equation*}
\end{corollary}

\begin{proof}
For the case $d=1$, recall that the cosine integral $\mathrm{Ci}(x)$ has the
series expansion 
\begin{equation*}
\mathrm{Ci}(x)=\gamma +\log x+\sum_{k=1}^{+\infty }\frac{(-x^{2})^{k}}{%
2k(2k)!},\quad x>0
\end{equation*}%
where $\gamma $ is the Euler--Mascheroni constant, from which our claim
easily follows.

For $d\geq 2$, $\Theta _{\sigma }$ is only defined as a tempered
distribution.
\end{proof}

\begin{corollary}
\label{bo2} The function $\Theta _{\sigma }\notin L^{p}({\mathbb{R}^{2d}})$,
for any $1\leq p\leq \infty $.
\end{corollary}

\begin{proof}
The case $p=\infty $ is already treated in Corollary \ref{Thetalinfty}. For $%
d\geq 2$ again we observe that $\Theta _{\sigma }$ is not defined as
function but only as a tempered distribution. For $d=1$, $1\leq p<\infty $,
the claim follows by the expression \eqref{Ci}. Indeed, choose $0<\epsilon
<\pi /2$, then $|\mathrm{Ci}(x)|\geq |\mathrm{Ci}(\epsilon )|$, for $%
0<x<\epsilon $, so that 
\begin{align*}
\int_{\mathbb{R}^{2}}|\Theta _{\sigma }(\zeta _{1},\zeta _{2})|^{p}\,d\zeta
_{1}d\zeta _{2}& \geq 2\int_{|\zeta _{1}\zeta _{2}|<\frac{\epsilon }{4\pi }}|%
\mathrm{Ci}(4\pi |\zeta _{1}\zeta _{2}|)|^{p}\,d\zeta d\zeta _{2} \\
& \geq C_{p}meas\{(\zeta _{1},\zeta _{2})\,:\,|\zeta _{1}\zeta _{2}|<\frac{%
\epsilon }{4\pi }\}=+\infty ,
\end{align*}%
for a suitable constant $C_{p}>0$.
\end{proof}

Since $\mathcal{F} L^1\subset L^\infty$, the Wiener amalgam space $W(%
\mathcal{F} L^1, L^\infty)$ is included in $L^\infty_{loc}$. This proves our
claim:

\begin{corollary}
\label{bo1} The function $\Theta_{\sigma} \notin W(\mathcal{F} L^1,L^\infty)({%
\mathbb{R}^{2d}})$ or, equivalently, $\Theta\notin M^{1,\infty}({\mathbb{R}%
^{2d}})$.
\end{corollary}

\section{Cohen Kernels in modulation and Wiener spaces}

In this section we focus on other members of the Cohen class, introduced by
Cohen in \cite{Cohen1}, which define, for $\tau \in \lbrack 0,1]$, the
(cross-)$\tau $-Wigner distributions 
\begin{equation}
W_{\tau }(f,g)(x,\omega )=\int_{\mathbb{R}^{d}}e^{-2\pi iy\zeta }f(x+\tau y)%
\overline{g(x-(1-\tau )y)}\,dy\quad f,g\in \mathcal{S}(\mathbb{R}^{d}).
\label{tauwig}
\end{equation}%
Such distributions enter in the definition of the $\tau -$pseudodifferential
operators as follows 
\begin{equation}
\langle \limfunc{Op}\nolimits_{\mathrm{\tau }}(a)f,g\rangle =\langle
a,W_{\tau }(g,f)\rangle \quad f,g\in \mathcal{S}(\mathbb{R}^{d}).
\label{tauweak}
\end{equation}%
It is then natural to investigate the time-frequency properties of such
kernels and compare to the corresponding Weyl and Born-Jordan ones. The
Cohen class consists of elements of the type 
\begin{equation*}
M(f,f)(x,\omega )=W(f,f)\ast \sigma
\end{equation*}%
where $\sigma \in \mathcal{S}^{\prime }({\mathbb{R}^{2d}})$ is called the
Cohen kernel. When $\sigma =\delta $, then $M(f,f)=W(f,f)$ and we come back
to the Wigner distribution. When $\sigma =\Theta _{\sigma }$, then $%
M(f,f)=Q(f,f)$, that is the Born-Jordan distribution. The $\tau $-Wigner
function $W_{\tau }(f,f)$ belongs to the Cohen class for every $\tau \in
\lbrack 0,1]$, as proved in \cite[Proposition 5.6]{bogetal}: 
\begin{equation*}
W_{\tau }(f,f)=W(f,f)\ast \sigma _{\tau },
\end{equation*}%
where 
\begin{equation*}
\sigma _{\tau }(x,\omega )=\frac{2^{d}}{|2\tau -1|^{d}}e^{2\pi i\frac{2}{%
2\tau -1}x\omega },\quad \tau \not=\frac{1}{2}
\end{equation*}%
and $\sigma _{1/2}=\delta $ (the case of the Wigner distribution, as already
observed).

In what follows we study the properties of the Cohen kernels $\sigma_{\tau }$
in the realm of modulation and Wiener amalgam spaces. As we shall see, the
Born-Jordan kernel and the Wigner one display similar time-frequency
properties and are locally worse than the kernels $\sigma_{\tau }$, $\tau\not=
1/2$.

\begin{proposition}
\label{tauW} We have, for every $\tau\in[0,1]\setminus \{1/2\}$, 
\begin{equation*}
\sigma_{\tau }\in W(\mathcal{F} L^1,L^\infty)({\mathbb{R}^{2d}})\cap
M^{1,\infty}({\mathbb{R}^{2d}}).
\end{equation*}
\end{proposition}

\begin{proof}
We exploit the dilation properties for Wiener spaces (cf.\ \cite[Lemma 3.2]%
{sugitomita2} and its generalization in \cite[Corollary 3.2]{CNJFA2008}):
for $A=\lambda I$, $\lambda>0$, 
\begin{equation}  \label{dillambda}
\|f(A\,\cdot)\|_{W(\mathcal{F}L^p,L^q)}\leq C
\lambda^{d\left(\frac1p-\frac1q-1\right)}(\lambda^2+1)^{d/2} \|f\|_{W(%
\mathcal{F}L^p,L^q)}.
\end{equation}
Using the dilation relations for Wiener amalgam spaces \eqref{dillambda} for 
$\lambda=\sqrt{t}$, $0<t<1/2$, $p=1$, $q=\infty$, we obtain 
\begin{equation*}
\|e^{\pm 2\pi i \zeta_1\zeta_2 t}\|_{W(\mathcal{F} L^1,L^\infty)}\leq C
\|e^{\pm 2\pi i \zeta_1\zeta_2 }\|_{W(\mathcal{F} L^1,L^\infty)}
\end{equation*}
with constant $C>0$ independent on the parameter $t$. For $t=\frac
2{2\tau-1} $, when $\tau>1/2$ and $t=-\frac 2{2\tau-1}$, when $0\leq \tau
<1/2$, we obtain that $\sigma_{\tau }\in W(\mathcal{F} L^1,L^\infty)({\mathbb{R}%
^{2d}})$. Now, an easy computation gives 
\begin{equation*}
\mathcal{F} \sigma_{\tau }(\zeta_1,\zeta_2)=e^{-\pi i (2\tau-1)\zeta_1\zeta_2},
\end{equation*}
so that, using $\mathcal{F} M^{1,\infty}({\mathbb{R}^{2d}})=W(\mathcal{F}
L^1,L^\infty)({\mathbb{R}^{2d}})$ and repeating the previous argument we
obtain $\sigma_{\tau }\in M^{1,\infty}({\mathbb{R}^{2d}})$ for every $\tau\in
[0,1]\setminus \{1/2\}$.
\end{proof}

The case $\tau=1/2$ is different. Indeed, $\sigma_{1/2}=\delta$ and for any
fixed $g\in\mathcal{S}({\mathbb{R}^{2d}})\setminus \{0\}$ the STFT $V_g
\delta$ is given by 
\begin{equation*}
V_g \delta (z,\zeta)=\langle \delta, M_\zeta T_z g\rangle=\overline{g(-z)},
\end{equation*}
that yields $\delta\in M^{1,\infty}({\mathbb{R}^{2d}})\setminus W(\mathcal{F}
L^1, L^\infty)({\mathbb{R}^{2d}})$.

The Born-Jordan kernel $\Theta_{\sigma}$ behaves analogously. Indeed, using
Proposition \ref{sincxpwiener} and Corollary \ref{bo1}, we obtain 
\begin{equation*}
\Theta_{\sigma} \in M^{1,\infty}({\mathbb{R}^{2d}})\setminus W(\mathcal{F}
L^1, L^\infty)({\mathbb{R}^{2d}}).
\end{equation*}
Those distributions can be used in the definition of the $\tau$%
-pseudodifferential operators

\section{Symbols in modulation spaces}

This section is focused on the proof of Theorem \ref{Charpseudo}. We first
demonstrate the sufficient boundedness conditions.

\begin{theorem}
\label{Charpseudosuff} Assume that $1\leq p,q,r_1,r_2\leq \infty$. Then the
pseudodifferential operator $\limfunc{Op}\nolimits_{\mathrm{BJ}}(a)$, from $%
\mathcal{S}(\mathbb{R}^d)$ to $\mathcal{S}^{\prime }(\mathbb{R}^d)$, having
symbol $a \in M^{p,q}(\mathbb{R}^{2d})$, extends uniquely to a bounded
operator on $\mathcal{M}^{r_1,r_2}(\mathbb{R}^d)$, with the estimate %
\eqref{stimaA} and the indices' conditions \eqref{indicitutti} and %
\eqref{indiceq}.
\end{theorem}

The result relies on a thorough understanding the action of the mapping 
\begin{equation}  \label{map00}
A: \, a\longmapsto a\ast\Theta_{\sigma},
\end{equation}
which gives the Weyl symbol of an operator with Born-Jordan symbol $a$, on
modulation space s.

\begin{prop}
\label{Abound} For every $1\leq p,q\leq\infty$, the mapping $A$ in %
\eqref{map00}, defined initially on $\mathcal{S}^{\prime }({\mathbb{R}^{2d}}%
) $, restricts to a linear continuous map on $M^{p,q} ({\mathbb{R}^{2d}})$,
i.e., there exists a constant $C>0$ such that 
\begin{equation}  \label{contA}
\|Aa\|_{M^{p,q}}\leq C\|a\|_{M^{p,q}}.
\end{equation}
\end{prop}

\begin{proof}
By Proposition \eqref{sincxpwiener}, the function $\Theta\in W(\mathcal{F}
L^1,L^\infty)({\mathbb{R}^{2d}})$. Its symplectic Fourier transform \, $%
\Theta_{\sigma}$ belongs to $\mathcal{F}_{\sigma} W(\mathcal{F} L^1,L^\infty)({%
\mathbb{R}^{2d}})= M^{1,\infty}({\mathbb{R}^{2d}})$. Now, for every $1\leq
p,q\leq\infty$, the convolution relations for modulation space s %
\eqref{mconvm} give 
\begin{equation*}
M^{p,q}({\mathbb{R}^{2d}})\ast M^{1,\infty}({\mathbb{R}^{2d}})
\hookrightarrow M^{p,q}({\mathbb{R}^{2d}})
\end{equation*}
and this shows the claim \eqref{contA}.
\end{proof}

\begin{proof}[Proof of Theorem \protect\ref{Charpseudosuff}.]
Assume $a \in M^{p,q}(\mathbb{R}^{2d})$, then Proposition \ref{Abound}
proves that $Aa= a\ast\Theta_{\sigma} \in M^{p,q}(\mathbb{R}^{2d})$ as well.
We next write $\limfunc{Op}\nolimits_{\mathrm{BJ}}(a)=\limfunc{Op}\nolimits_{%
\mathrm{W}}(Aa)$ and use the sufficient conditions for Weyl operators in 
\cite[Theorem 5.2]{cordero2}: if the Weyl symbol $Aa$ is in $M^{p,q}(\mathbb{%
R}^{2d})$, then $\limfunc{Op}\nolimits_{\mathrm{W}}(Aa)$ extends to a
bounded operator on $\mathcal{M}^{r_1,r_2}(\mathbb{R}^d)$, with 
\begin{equation*}
\|\limfunc{Op}\nolimits_{\mathrm{BJ}}(a) f\|_{\mathcal{M}^{r_1,r_2}}=\|%
\limfunc{Op}\nolimits_{\mathrm{W}}(Aa)f\|_{\mathcal{M}^{r_1,r_2}} \lesssim
\|Aa\|_{ M^{p,q}}\|f\|_{\mathcal{M}^{r_1,r_2}}
\end{equation*}
where the indices $r_1,r_2,p,q$ satisfy \eqref{indicitutti} and %
\eqref{indiceq}. The inequality \eqref{contA} then provides the claim.
\end{proof}

The necessary conditions of Theorem \ref{Charpseudo} require some
preliminaries.

We reckon the adjoint operator $\limfunc{Op}\nolimits_{\mathrm{BJ}}(a)^*$ of
a Born-Jordan operator $\limfunc{Op}\nolimits_{\mathrm{BJ}}(a)$ using the
connection with Weyl operators. Recall that $\limfunc{Op}\nolimits_{\mathrm{W}}(b)^\ast= \limfunc{Op}\nolimits_{\mathrm{W}}(\bar{b})$ \cite{hormander-book}, so that 
\begin{equation*}
\limfunc{Op}\nolimits_{\mathrm{BJ}}(a)^*=\limfunc{Op}\nolimits_{\mathrm{W}}(a\ast\Theta_{\sigma})^\ast=\limfunc{Op}\nolimits_{\mathrm{W}}(\overline{a\ast\Theta_{\sigma}})=\limfunc{Op}\nolimits_{\mathrm{W}}(\bar{a}\ast\bar{\Theta_{\sigma}})=\limfunc{Op}\nolimits_{\mathrm{W}}(\bar{a}\ast\Theta_{\sigma})=\limfunc{Op}\nolimits_{\mathrm{BJ}}(\bar{a})
\end{equation*}
because $\Theta$ is an even real-valued function. Hence the adjoint of a
Born-Jordan operator $\limfunc{Op}\nolimits_{\mathrm{BJ}}(a)$ with symbol $a$
is the Born-Jordan operator having symbol $\bar{a}$ (the complex-conjugate
of $a$). This nice property is the key argument for the following auxiliary
result, already obtained for the case of Weyl operators in \cite[Lemma 5.1]%
{cordero2}. The proof uses the same pattern as the former result and hence
is omitted.

\begin{lemma}
\label{dualita2} Suppose that, for some $1\leq p,q, r_1,r_2\leq\infty$, the
following estimate holds: 
\begin{equation*}
\|\limfunc{Op}\nolimits_{\mathrm{BJ}}(a) f\|_{M^{r_1,r_2}}\leq
C\|a\|_{M^{p,q}}\|f\|_{M^{r_1,r_2}},\ \ \forall a\in \mathcal{S}(\mathbb{R}%
^{2d}),\ \forall f\in \mathcal{S}(\mathbb{R}^d).
\end{equation*}
Then the same estimate is satisfied with $r_1,r_2$ replaced by $r_1^{\prime
},r_2^{\prime }$ (even if $r_1=\infty$ or $r_2=\infty$).
\end{lemma}

The above instruments let us show the necessity of \eqref{indicitutti} and %
\eqref{stimanew}.

\begin{theorem}
\label{1-2} Suppose that, for some $1\leq p,q,r_1,r_2\leq\infty$, $C>0$ the
estimate 
\begin{equation}  \label{stimamod}
\|\limfunc{Op}\nolimits_{\mathrm{BJ}}(a) f\|_{M^{r_1,r_2}}\leq C
\|a\|_{M^{p,q}}\|f\|_{M^{r_1,r_2}}\qquad \forall a\in\mathcal{S}(\mathbb{R}%
^{2d}),\ f\in\mathcal{S}(\mathbb{R}^d)
\end{equation}
holds. Then the constraints in \eqref{indicitutti} and \eqref{stimanew} must
hold.
\end{theorem}

\begin{proof}
The estimate \eqref{stimamod} can be written as 
\begin{equation*}
|\langle a,Q(f,g)\rangle |\leq C\Vert a\Vert _{M^{p,q}}\Vert f\Vert
_{M^{r_{1},r_{2}}}\Vert g\Vert _{M^{r_{1}^{\prime },r_{2}^{\prime }}}\qquad
\forall a\in \mathcal{S}(\mathbb{R}^{2d}),\ f,g\in \mathcal{S}(\mathbb{R}%
^{d})
\end{equation*}%
which is equivalent to 
\begin{equation*}
\Vert Q(f,g)\Vert _{M^{p^{\prime },q^{\prime }}}\leq C\Vert f\Vert
_{M^{r_{1},r_{2}}}\Vert g\Vert _{M^{r_{1}^{\prime },r_{2}^{\prime }}}\qquad
\forall f,g\in \mathcal{S}(\mathbb{R}^{d}).
\end{equation*}%
Now, one should test this estimate on families of functions $f_{\lambda }$, $%
g_{\lambda }$ such that $Q(f_{\lambda },g_{\lambda })$ is concentrated
inside the hyperbola $|x\cdot \omega |<1$ (say), see Figure \ref{figura},
where $\theta \asymp 1$, so that the left-hand side is comparable to $\Vert
W(f_{\lambda },g_{\lambda })\Vert _{_{M^{p^{\prime },q^{\prime }}}}$ and can
be estimated from below.

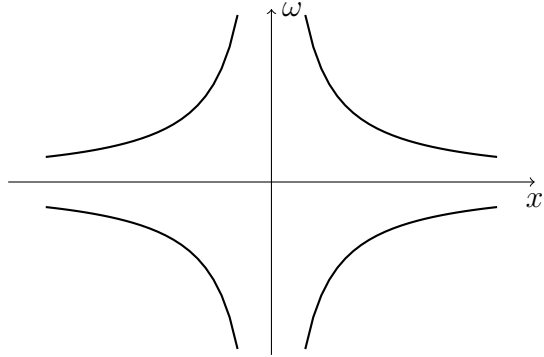
\begin{figure}[b]
\begin{tikzpicture}
\draw[->] (-3.5,0) -- (3.5,0) node[below]{$x$};
\draw[->] (0,-2.3) -- (0,2.3) node[right]{$\omega$};
\draw[thick, domain=0.45:3] plot (\x, {1/\x}) node[above] {};
\draw[thick, domain=0.45:3] plot (\x, {-1/\x}) node[above]{};
\draw[thick, domain=-3:-0.45] plot (\x, {1/\x}) node[above]{};
\draw[thick, domain=-3:-0.45] plot (\x, {-1/\x}) node[above]{};
\end{tikzpicture}
\caption{The region $|x\cdot\protect\omega|<1$ ($d=1$).}
\label{figura}
\end{figure}

The choice $f_{\lambda}(x)=g_{\lambda}(x)=e^{-\pi\lambda|x|^2}$, provides the
estimate \eqref{indicitutti} when $\lambda\to+\infty$. Indeed in this case
we argue exactly as in the proof of \cite[Theorem 1.4]{cgn2}. We recall this
pattern, useful also for other cases. By \eqref{eqa0-zero} we obtain the
estimate 
\begin{equation}  \label{ef1}
\|\varphi_{\lambda}\|_{M^{r_1,r_2}}\|\varphi_{\lambda}\|_{M^{r^{\prime
}_1,r^{\prime }_2}}\lesssim \lambda^{-\frac d {2r^{\prime }_2}}
\lambda^{-\frac d{2r_2}}.
\end{equation}
We gauge from below the norm $\|Q(\varphi_{\lambda},\varphi_{\lambda})\|_{M^{p^{%
\prime },q^{\prime }}}$ as follows. By taking the symplectic Fourier
transform and using Lemma \ref{lemma5.1} and the product property %
\eqref{product} we have 
\begin{align*}
\|Q(\varphi_{\lambda},\varphi_{\lambda})\|_{M^{p^{\prime },q^{\prime
}}}&=\|\Theta_{\sigma} \ast W(\varphi_{\lambda},\varphi_{\lambda})\|_{M^{p^{\prime
},q^{\prime }}} \\
&\asymp \|\Theta \mathcal{F}_{\sigma}[W(\varphi_{\lambda},\varphi_{\lambda})]\|_{W(%
\mathcal{F} L^{p^{\prime }},L^{q^{\prime }})} \\
&\gtrsim \|\Theta(\zeta_1,\zeta_2) \chi(\zeta_1\zeta_2)\mathcal{F}_{\sigma}[
W(\varphi_{\lambda},\varphi_{\lambda})]\|_{W(\mathcal{F} L^{p^{\prime
}},L^{q^{\prime }})}
\end{align*}
for any $\chi\in C^\infty_c(\mathbb{R})$. Choosing $\chi$ supported in the
interval $[-1/4,1/4]$ and $=1$ in the interval $[-1/8,1/8]$, we write 
\begin{equation*}
\chi(\zeta_1 \zeta_2)=\chi(\zeta_1 \zeta_2)
\Theta(\zeta_1,\zeta_2)\Theta^{-1}(\zeta_1,\zeta_2)\tilde{\chi}(\zeta_1
\zeta_2),
\end{equation*}
with $\tilde{\chi}\in C^\infty_c(\mathbb{R})$ supported in $[-1/2,1/2]$ and $%
\tilde{\chi}=1$ on $[-1/4,1/4]$, therefore on the support of $\chi$. Since
by Lemma \ref{lemma5.1} the function $\Theta^{-1}(\zeta_1,\zeta_2)\tilde{\chi%
}(\zeta_1 \zeta_2)$ belongs to $W(\mathcal{F} L^1,L^\infty)$, again by the
product property the last expression is estimated from below as 
\begin{equation*}
\gtrsim \| \chi(\zeta_1\zeta_2)\mathcal{F}_{\sigma}[ W(\varphi_{\lambda},%
\varphi_{\lambda})]\|_{W(\mathcal{F} L^{p^{\prime }},L^{q^{\prime }})}.
\end{equation*}
Consider a function $\psi\in C^\infty_c(\mathbb{R}^d)\setminus\{0\}$,
supported in $[-1/4,1/4]$. Using 
\begin{equation*}
|\zeta_1 \zeta_2|\leq\frac{1}{2}(|\sqrt{\lambda}\zeta_1|^2+|\sqrt{\lambda}%
^{-1}\zeta_2|^2)
\end{equation*}
we see that $\chi(\zeta_1 \zeta_2)=1$ on the support of $\psi(\sqrt{\lambda}
\zeta_1)\psi(\sqrt{\lambda}^{-1}\zeta_2)$, for every $\lambda>0$.

Then, we can write 
\begin{equation*}
\psi(\sqrt{\lambda} \zeta_1)\psi(\sqrt{\lambda}^{-1}\zeta_2)=\chi(\zeta_1
\zeta_2)\psi(\sqrt{\lambda} \zeta_1)\psi(\sqrt{\lambda}^{-1}\zeta_2)
\end{equation*}
and by Lemma \ref{lemma5.2} we also infer 
\begin{equation*}
\|\psi(\sqrt{\lambda} \zeta_1)\psi(\sqrt{\lambda}^{-1}\zeta_2)\|_{W(\mathcal{%
F} L^1,L^\infty)}\lesssim 1
\end{equation*}
so that we can continue the above estimate as 
\begin{equation*}
\gtrsim \|\psi(\sqrt{\lambda} \zeta_1)\psi(\sqrt{\lambda}^{-1}\zeta_2)%
\mathcal{F}_{\sigma}[ W(\varphi_{\lambda},\varphi_{\lambda})]\|_{W(\mathcal{F}
L^{p^{\prime }},L^{q^{\prime }})}.
\end{equation*}
Using (see e.g.\ \cite[Formula (4.20)]{book}) 
\begin{equation}  \label{wignerdil}
W(\varphi_{\lambda},\varphi_{\lambda})(x,\omega)=2^{\frac d2} \lambda^{-\frac
d2} \varphi(\sqrt{2\lambda}\, x)\varphi(\sqrt{\frac 2\lambda}\, \omega),
\end{equation}
we calculate 
\begin{equation*}
\mathcal{F}_{\sigma}[ W(\varphi_{\lambda},\varphi_{\lambda})](\zeta_1,\zeta_2)=
2^{\frac d2} \lambda^{-\frac d2} \varphi((\sqrt{2\lambda})^{-1}\,
\zeta_2)\varphi(\sqrt{\frac{\lambda }2}\, \zeta_1),
\end{equation*}
so that 
\begin{align*}
& \|\psi(\sqrt{\lambda}\zeta_1)\psi(\sqrt{\lambda}^{-1}\zeta_2)\mathcal{F}_%
\sigma[ W(\varphi_{\lambda},\varphi_{\lambda})]\|_{W(\mathcal{F} L^{p^{\prime
}},L^{q^{\prime }})} \\
&=2^{-d/2} \lambda^{-\frac d2}\|\psi(\sqrt{\lambda} \zeta_1)\varphi((1/\sqrt{%
2})\sqrt{\lambda}\, \zeta_1) \|_{W(\mathcal{F} L^{p^{\prime }},L^{q^{\prime
}})}\| \psi(\sqrt{\lambda}^{-1}\zeta_2)\varphi((\sqrt{2\lambda})^{-1}\,
\zeta_2)\|_{{W(\mathcal{F} L^{p^{\prime }},L^{q^{\prime }})}}.
\end{align*}
By Lemma \ref{lemma5.2} we can estimate the last expression so that 
\begin{equation*}
\|Q(\varphi_{\lambda},\varphi_{\lambda})\|_{M^{p^{\prime },q^{\prime }}}\gtrsim
\lambda^{-d+\frac d{2p^{\prime }}+\frac d{2q^{\prime }}}\quad \mathrm{as}\
\lambda\to+\infty.
\end{equation*}
Finally, using the estimate \eqref{ef1} we infer \eqref{indicitutti}.

We now prove that $\max \{1/r_{1},1/r_{1}^{\prime }\}\leq 1/q+1/p$. If we
show the estimate $1/r_{1}\leq 1/q+1/p$, then the constraint $%
1/r_{1}^{\prime }\leq 1/q+1/p$ follows by the duality argument of Lemma \ref%
{dualita2}. To reach this goal, we consider $f_{\lambda }=\varphi $
(independent of the parameter $\lambda $) and $g=\varphi _{\lambda }$ as
before and use the previous pattern for these families of functions, in the
case $\lambda \rightarrow 0^{+}$. By \eqref{eqa0-zero} the upper estimate
becomes 
\begin{equation}
\Vert \varphi \Vert _{M^{r_{1},r_{2}}}\Vert \varphi _{\lambda }\Vert
_{M^{r_{1}^{\prime },r_{2}^{\prime }}}\lesssim \lambda ^{-\frac{d}{%
2r_{1}^{\prime }}}.  \label{ef2}
\end{equation}%
The same arguments as before let us write 
\begin{equation*}
\Vert Q(\varphi ,\varphi _{\lambda })\Vert _{M^{p^{\prime },q^{\prime
}}}\gtrsim \Vert \psi (\sqrt{\lambda }\zeta _{1})\psi (\sqrt{\lambda }%
^{-1}\zeta _{2})\mathcal{F}_{\sigma }[W(\varphi ,\varphi _{\lambda })]\Vert
_{W(\mathcal{F}L^{p^{\prime }},L^{q^{\prime }})},
\end{equation*}%
where $\mathcal{F}_{\sigma }[W(\varphi ,\varphi _{\lambda })]$ is computed
in \eqref{cfwfflsig}. Observe that, given any $F\in W(\mathcal{F}%
L^{p^{\prime }},L^{q^{\prime }})$, 
\begin{align*}
\Vert e^{\pi i\frac{\lambda -1}{%
\lambda +1}\zeta
_{1}\zeta _{2}}F(\zeta _{1},\zeta _{2})\Vert _{W(\mathcal{F}L^{p^{\prime
}},L^{q^{\prime }})}& \gtrsim \Vert e^{-\pi i\frac{\lambda -1}{%
\lambda +1}\zeta _{1}\zeta _{2}}e^{\pi i\frac{\lambda -1}{%
\lambda +1}\zeta _{1}\zeta _{2}}F(\zeta _{1},\zeta
_{2})\Vert _{W(\mathcal{F}L^{p^{\prime }},L^{q^{\prime }})} \\
& =\Vert F(\zeta _{1},\zeta _{2})\Vert _{W(\mathcal{F}L^{p^{\prime
}},L^{q^{\prime }})},
\end{align*}%
because $\Vert e^{-\pi i\frac{\lambda -1}{%
\lambda +1}%
\zeta _{1}\zeta _{2}}\Vert _{W(\mathcal{F}L^{1},L^{\infty })}\leq C,$ for
every $\lambda >0$ by \cite[Proposition 3.2]{cgn2}. So, writing 
\begin{equation*}
c_{\lambda }=\frac{1}{(\lambda +1)^{\frac{d}{2}}}
\end{equation*}%
(notice $c_{\lambda }\rightarrow 1$ for $\lambda \rightarrow 0^{+}$) we
are reduced to  
\begin{equation*}
\Vert Q(\varphi ,\varphi _{\lambda })\Vert _{M^{p^{\prime },q^{\prime
}}}\gtrsim c_{\lambda }\Vert \psi (\sqrt{\lambda }\zeta _{1})e^{- \frac{\pi
\lambda }{\lambda +1}\zeta _{1}^{2}}\Vert _{W(
\mathcal{F}L^{p^{\prime }},L^{q^{\prime }})}\Vert \psi (\sqrt{\lambda }
^{-1}\zeta _{2})e^{-\frac{\pi }{\lambda +1}\zeta
_{2}^{2}}\Vert _{W(\mathcal{F}L^{p^{\prime }},L^{q^{\prime }})}.
\end{equation*}
By Lemma \ref{lemma5.2} we can estimate, for $\lambda \rightarrow 0^{+}$, 
\begin{equation*}
\Vert \psi (\sqrt{\lambda }\zeta _{1})e^{- \frac{\pi
\lambda }{\lambda +1}\zeta _{1}^{2}}\Vert _{W(\mathcal{F}L^{p^{\prime
}},L^{q^{\prime }})}=\Vert \psi (\sqrt{\lambda }\zeta _{1})e^{- \frac{\pi }{\lambda +1}(\sqrt{\lambda }\zeta _{1})^{2}}\Vert
_{W(\mathcal{F}L^{p^{\prime }},L^{q^{\prime }})}\asymp \lambda ^{-\frac{d}{
2q^{\prime }}},
\end{equation*}
whereas 
\begin{align*}
\Vert \psi (\sqrt{\lambda }^{-1}\zeta _{2}))e^{-\frac{\pi }{\lambda +1}\zeta
_{2}^{2}}\Vert _{W(\mathcal{F}L^{p^{\prime }},L^{q^{\prime }})}& =\lambda ^{\frac d2}(\lambda+1)^{\frac d2}\Vert \int \hat{\psi}(\sqrt{\lambda}(\zeta _{2} -\eta ))e^{- \pi (\lambda +1) |\eta |^{2}}\,d\eta \Vert _{L^{p^{\prime }}} \\
& =\lambda ^{\frac d2}(\lambda+1)^{\frac d2}\lambda ^{-\frac d{2p'}}\Vert \int \hat{\psi}(x-\sqrt{\lambda}\eta
))e^{-\pi(\lambda+1) |\eta |^{2}}\,d\eta \Vert _{L^{p^{\prime }}} \\
& =(\lambda+1)^{\frac d2}\lambda ^{-\frac d{2p'}}\Vert \int \hat{\psi}(x-t)e^{-\pi\frac{\lambda+1}{\lambda} |t|^{2}}dt\Vert _{L^{p^{\prime }}} \\
& =\lambda ^{\frac d2-\frac d{2p^{\prime }}}\Vert \hat{\psi}\ast K_{1/\sqrt{\lambda }
}\Vert _{L^{p^{\prime }}} \\
& \sim \lambda^{\frac d2-\frac d{2p^{\prime }}}\Vert \hat{\psi}\Vert _{p^{\prime }},\,\,
\mbox{as}\,\,\lambda \rightarrow 0^{+}
\end{align*}
where $K_{1/\sqrt{\lambda }}(\zeta _{2})=\lambda ^{-\frac d2}(\lambda +1)^{\frac d2}e^{-\frac{\pi (\lambda+1)}{\lambda}|\zeta_2|^2}$, $\lambda \rightarrow 0^{+}$, is an approximate
identity. So that 
\begin{equation*}
\lambda ^{-\frac{d}{2r_{1}^{\prime }}}\gtrsim \lambda ^{-\frac{d}{2q^{\prime
}}}\lambda ^{\frac{d}{2p}}
\end{equation*}%
and, for $\lambda \rightarrow 0^{+}$, we obtain 
\begin{equation*}
\frac{1}{r_{1}}\leq \frac{1}{q}+\frac{1}{p},
\end{equation*}%
as desired. 

It remains to prove that $\max\{1/r_2,1/r_2^{\prime }\}\leq 1/q+1/p$. Again,
it is enough to show that $1/r_2\leq 1/q+1/p$ and invoke Lemma \ref{dualita2}
for $1/r_2^{\prime }\leq 1/q+1/p$.

An explicit computation (see \cite[Proposition 5.3]{cordero2}) shows that 
\begin{equation}  \label{e0}
\mathcal{F}^{-1} \limfunc{Op}\nolimits_{\mathrm{W}}(\sigma)\mathcal{F}=%
\limfunc{Op}\nolimits_{\mathrm{W}}(\sigma\circ J),
\end{equation}
where $J(x,\omega)=(\omega,-x)$ as defined in \eqref{J} (this is also a
consequence of the intertwining property of the metaplectic operator $%
\mathcal{F}$ with the Weyl operator $\limfunc{Op}\nolimits_{\mathrm{W}%
}(\sigma)$ \cite[Corollary 221]{Birkbis}).


Now, observing that $\Theta_{\sigma}\circ J= \Theta_{\sigma}$, we obtain 
\begin{align*}
(a\ast\Theta_{\sigma})(J z)&=\int_{{\mathbb{R}^{2d}}} a(u)\Theta_{\sigma}(J
z-u)\,du=\int_{{\mathbb{R}^{2d}}} a(u) \Theta_{\sigma}(J(z-J^{-1}u)) du \\
&=\int_{{\mathbb{R}^{2d}}} a(u) \Theta_{\sigma}(z-J^{-1}u) du=\int_{{\mathbb{R}%
^{2d}}} a(Ju)\Theta_{\sigma}(z-u)\,du \\
&=(a\circ J)\ast\Theta_{\sigma}(z).
\end{align*}
The previous computations together with \eqref{e0} gives 
\begin{equation*}
\mathcal{F}^{-1} \limfunc{Op}\nolimits_{\mathrm{BJ}}(a) \mathcal{F}=\mathcal{%
F}^{-1} \limfunc{Op}\nolimits_{\mathrm{BJ}}(a\circ J) \mathcal{F}.
\end{equation*}
On the other hand, the map $a\longmapsto a\circ J$ is an isomorphism of $%
M^{p,q}$, so that \eqref{stimamod} is in fact equivalent to 
\begin{equation}  \label{stimamod2}
\|\limfunc{Op}\nolimits_{\mathrm{BJ}}(a) f\|_{W(\mathcal{F}
L^{r_1},L^{r_2})}\lesssim \|a\|_{M^{p,q}}\|f\|_{W(\mathcal{F}
L^{r_1},L^{r_2})}\qquad \forall a\in\mathcal{S}(\mathbb{R}^{2d}),\ f\in%
\mathcal{S}(\mathbb{R}^d).
\end{equation}
The estimate \eqref{stimamod2} can be written as 
\begin{equation*}
|\langle a, Q(f,g)\rangle | \leq C \|a\|_{M^{p,q}}\|f\|_{W(\mathcal{F}
L^{r_1},L^{r_2})}\|g\|_{W(\mathcal{F} L^{r^{\prime }_1},L^{r^{\prime
}_2})}\qquad \forall a\in\mathcal{S}(\mathbb{R}^{2d}),\ f,g\in\mathcal{S}(%
\mathbb{R}^d)
\end{equation*}
which is equivalent to 
\begin{equation*}
\|Q(f,g)\|_{M^{p^{\prime },q^{\prime }}} \leq C \|f\|_{W(\mathcal{F}
L^{r_1},L^{r_2})}\|g\|_{W(\mathcal{F} L^{r^{\prime }_1},L^{r^{\prime
}_2})}\qquad \forall f,g\in\mathcal{S}(\mathbb{R}^d).
\end{equation*}
Now, taking $f=\varphi$ and $g=\varphi_{\lambda}$ as before, we observe that,
for $\lambda\to0^+$, by \eqref{eqa0-zero}, 
\begin{equation*}
\|\varphi_{\lambda}\|_{W(\mathcal{F} L^{r^{\prime }_1},L^{r^{\prime
}_2})}\asymp \lambda^{-\frac{d}{2}}\|\varphi_{1/\lambda}\|_{M^{r^{\prime
}_1,r^{\prime }_2}}\asymp\lambda^{-\frac{d}{2}+\frac{d}{2r_2}} =\lambda^{-%
\frac{d}{2r^{\prime }_2}}.
\end{equation*}
Arguing as in the previous case we obtain $1/r_2\leq 1/q+1/p$. This
concludes the proof.
\end{proof}

\section*{Acknowledgements}

The first author was partially supported by the Italian
Local Project \textquotedblleft Analisi di Fourier per equazioni alle
derivate parziali ed operatori pseudo-differenziali", funded by the
University of Torino, 2013. The second author was supported by the FWF grant
P 2773.

\end{document}